\documentclass[12pt]{amsart}
\usepackage{color}
\usepackage{axodraw4j}
\usepackage{pstricks}
\usepackage{amssymb}
\usepackage{mathtools}
\usepackage{txfonts}
\usepackage{eucal}
\usepackage{extarrows}
\usepackage{wasysym}
\usepackage[all]{xy}

\usepackage{mathrsfs} % pour créé le \mathscr

\usepackage{caption}
\usepackage{float} 
\usepackage{amsmath}
\usepackage{tikz}
\usepackage{cite}
\usepackage{graphicx}
\usepackage{float}

\usepackage{xspace}
\usepackage[a4paper,body={16.3cm,22.8cm},centering]{geometry}
\usepackage[colorlinks=true,pagebackref=true]{hyperref} 
\hypersetup{urlcolor=blue, citecolor=red , linkcolor= blue}
\usepackage{shuffle}
\font \eightrm=cmr8

\newcommand{\nc}{\newcommand}
\nc\smsc{0.8}%small scale

\def\oprec{\!\!\joinrel{\ocircle\hskip -12.5pt \prec}\,}
\def\soprec{\,\joinrel{\ocircle\hskip -6.7pt \prec}\,}

\nc\delete[1]{}
\nc{\mlabel}[1]{\label{#1}}  % Use this to suppress names
\nc{\mcite}[1]{\cite{#1}}  % Use this to suppress names
\nc{\mref}[1]{\ref{#1}}  % Use this to suppress names
\nc{\mbibitem}[1]{\bibitem{#1}} % Use this to show number name
\delete{
\nc{\mlabel}[1]{\label{#1}  % Use the next two lines to show names
{\hfill \hspace{1cm}{\small\tt{{\ }\hfill(#1)}}}}
\nc{\mcite}[1]{\cite{#1}{\small{\tt{{\ }(#1)}}}}  % Use this lines to show names
\nc{\mref}[1]{\ref{#1}{{\tt{{\ }(#1)}}}}  % Use this lines to show names
\nc{\mbibitem}[1]{\bibitem[\bf #1]{#1}} % Use this to show name
}
\delete{
\nc{\mop}[1]{\mathop{\hbox {\rm #1} }}
\nc{\smop}[1]{\mathop{\hbox {\eightrm #1} }}
\nc{\mopl}[1]{\mathop{\hbox {\rm #1} }\limits}
\nc{\smopl}[1]{\mathop{\hbox {\eightrm #1} }\limits}
\def \restr#1{\mathstrut_{\textstyle |}\raise-8pt\hbox{$\scriptstyle #1$}}
\def \srestr#1{\mathstrut_{\scriptstyle |}\hbox to
  -1.5pt{}\raise-4pt\hbox{$\scriptscriptstyle #1$}}
\nc{\wt}{\widetilde}
\nc{\wh}{\widehat}

}
\newtheorem{theorem}{Theorem}[section]
\newtheorem{definition}{Definition}[section]

\newtheorem{corollary}{Corollary}[section]
\newtheorem{proposition}{Proposition}[section]
\newtheorem{lemma}{Lemma}[section]
\newtheorem{remark}{Remark}[section]

\newtheorem{examples}{Examples}[section]
\numberwithin{equation}{section}
\newcommand\alphlist{a,b,c,d,e,f,g,h,i,j,k,l,m,n,o,p,q,r,s,t,u,v,w,x,y,z}
\newcommand\Alphlist{A,B,C,D,E,F,G,H,I,J,K,L,M,N,O,P,Q,R,S,T,U,V,W,X,Y,Z}
\newcommand\getcmds[3]{\expandafter\newcommand\csname #2#1\endcsname{#3{#1}}}
\makeatletter
\@for\x:=\alphlist\do{\expandafter\getcmds\expandafter{\x}{frak}{\mathfrak}}% \fraka,\frakb,...
\@for\x:=\Alphlist\do{\expandafter\getcmds\expandafter{\x}{frak}{\mathfrak}}% \frakA,\frakB,...
\makeatother

\nc{\bfk}{{\bf k}}
\nc{\sha}{\shuffle}

\nc{\id}{\mathrm{id}}
\nc{\Id}{\mathrm{Id}}
\nc{\lbar}[1]{\overline{#1}}
\nc{\ot}{\otimes}
\nc{\dep}{\mathrm{dep}}
\nc{\ver}{\mathrm{ver}}

\nc{\tred}[1]{\textcolor{red}{#1}} \nc{\tgreen}[1]{\textcolor{green}{#1}}
\nc{\tblue}[1]{\textcolor{blue}{#1}} \nc{\tpurple}[1]{\textcolor{purple}{#1}}
\nc{\tcyan}[1]{\textcolor{cyan}{#1}} % qian lan se
\nc{\tblk}[1]{\textcolor{black}{#1}}

\nc{\li}[1]{\tpurple{\underline{Li:}#1 }}
\nc{\liadd}[1]{\tpurple{#1}}
\nc{\xing}[1]{\tblue{\underline{Xing:}#1 }}
\nc{\yuan}[1]{\tred{\underline{Yuan:}#1 }}
\nc{\markus}[1]{\tred{\underline{Markus:} #1}}
\nc{\dominique}[1]{\tpurple{\underline{Dominique: }#1 }}
\long\def\ignore#1{}

\usetikzlibrary{calc,shapes.geometric}
\tikzset{
baseon/.style={baseline={($(#1)+(0,-0.58ex)$)}},
baseon/.default=current bounding box.center,
every picture/.style=baseon,
lst/.style={},
dst/.style={circle,inner sep=1pt,outer sep=0pt,fill,draw,dst2},
dst2/.style={fill=white},
ddst/.style={diamond,draw,inner sep=1pt},
eest/.style={ellipse,draw,inner sep=1pt,minimum size=2ex},
}

\def\zzz#1`#2...#3`#4...#5`#6@{%
--++(#1)%coordinate(#2)
node[dst,label={#5:$#6$},name=#2]{}
node[midway,\hbox{Aut}o,#3]{$#4$}
}
\def\ddd#1`#2`#3@{+(#1)node[ddst,name=#2]{$#3$}}
\def\eee#1`#2`#3@{+(#1)node[eest,name=#2]{$#3$}}
\def\xxx#1`#2@{node[midway,\hbox{Aut}o,inner sep=1pt,#1]{$#2$}}
\def\pp#1`#2`#3@{node[dst,label={#2:$#3$},pos=#1]{}}
\def\oo#1`#2`#3@{\path (o) node[dst,label={#2:$#3$},name=o,#1]{};}
\def\eoo#1`#2@{\node[eest,name=o,#1] at (o) {$#2$};}

\newif\ifshowjdq
\showjdqtrue
\newcommand\setXXclip[3]{%
\def\XXheight{#1}\def\XXdepth{#2}\def\XXwidth{#3}}
\setXXclip{1}{-0.5}{1.1}

\makeatletter
\newcommand\simra{\mathrel{\mathpalette\@verra\sim}}
\def\@verra#1#2{\lower.5\p@\vbox{\lineskiplimit\maxdimen \lineskip-.5\p@
\ialign{$\m@th#1\hfil##\hfil$\crcr#2\crcr\rightarrow\crcr}}}
\makeatother

\nc{\dnx}{\Delta_n A} \nc{\dx}{\Delta A} \nc{\dgp}{{\rm deg_{P}}}
\nc{\dgt}{{\rm deg_{T}}} \nc{\dg}{{\rm deg}} \nc{\ida}{ID($A$)} \nc{\tu}{\tilde{u}} \nc{\tv}{\tilde{v}}
\nc{\nr}{\calr_n} \nc{\nz}{\calz_n} \nc{\fun}{\cala_{n,d}}
 \nc{\fbase}{\calb} \nc{\LF}{\mathrm{RF}} \nc{\FFA}{\mathrm{LF}} \nc{\irr}{\mathrm{Irr}}
 \nc{\result}{\bfk\mathrm{Irr}(S_n)}  \nc{\I}{I_{\mathrm{ID},n}^0}
 \nc{\nrs}{\calr_n^\star} \nc{\ii}{\mathrm{I}} \nc{\iii}{\mathrm{II}}
\nc{\intl}{{\rm int}}\nc{\ws}[1]{{#1}}\nc{\deleted}[1]{\delete{#1}}\nc{\plas}{placements\xspace}

\nc{\bim}[1]{#1}  \nc{\shaop}{\sha_{\Omega}^{+}}  \nc{\shao}{\sha_{\Omega}}
\nc{\bbim}[2]{#1 #2} \nc{\bbbim}[2]{#1,\, #2} \nc{\RBF}{{\rm RBF}}
\nc{\frb}{F_{\RB}} \nc{\shaf}{\ssha_{\tiny{\Omega}}} \nc{\sham}{\blacklozenge_{\tiny{\Omega}}}
\nc{\lf}{\lfloor} \nc{\rf}{\rfloor} \nc{\shan}{\ssha_{\lambda}}
\nc{\rlex}{{\rm {lex}}} \nc{\bb}{\Box} \nc{\ra}{\rightarrow}
\nc{\e}{{\rm {e}}}
\nc{\DDF}{\mathrm{DD}(X,\,\Omega)}\nc{\DTF}{\mathrm{DT}(X,\,\Omega)} \nc{\DT}{\mathrm{DT}'(\Omega,\,V)}
\nc{\bra}{\mathrm{bra}} \nc{\bre}{\mathrm{bre}}
\nc{\dec}{\mathrm{dec}} \nc{\blacklozengew}{\blacklozenge_{w}}
\nc{\type}{\mathrm{type}}

\nc\calt{\cal{T}(X,\,\Omega)} \nc\caltn{\cal{T}_n(X,\,\Omega)}
\nc\calta{\cal{T}_0(X,\,\Omega)}
\nc\caltb{\cal{T}_1(X,\,\Omega)}
\nc\caltc{\cal{T}_2(X,\,\Omega)}
\nc\caltd{\cal{T}_3(X,\,\Omega)}
\nc\caltm{\cal{T}_m(X,\,\Omega)}
\nc\caltx{\cal{T}(X)}
\nc\calf{\cal{F}(X,\,\Omega)}
\nc\fram{\frak{M}(\Omega,\, X)}
\nc\shaw{\sha^{NC}_w(\Omega,\, X)}
\nc\dw{\blacklozenge_w} \nc\dl{\blacklozenge_\ell}
\nc\shal{\sha^{NC}_\ell(X,\, \Omega)} \nc\shav{\sha^{NC}_w(\Omega,\, V)} \nc\shat{\sha^{NC,1}_w(\Omega,\, T^{+}(V))}
\nc{\cfo}{\cal{F}(X,\,\Omega)}
\nc{\sh}{\rm{Sh}}
\nc{\lar}{\varinjlim}
\def\cxo#1#2;{\cal{#1}#2\XO}
\nc\lrf[2]{B_{#2}^+(#1)}
\nc{\fd}{\mathrm{\text{typed angularly decorated planar rooted trees}}}
\nc{\rb}{\mathrm{RBFWs}} \nc{\dfw}{\mathrm{DFW{(X)}}} \nc{\tfw}{\mathrm{TFW{(X)}}}
\nc{\tfv}{\mathrm{TFW{(V)}}}

\def\Ve#1,#2,#3;{\vee_{#1,\,(#2,\,#3)}}
\def\bigv#1;#2;#3;{\bigvee\nolimits_{#1}^{#2;\,#3}}
\nc\rjt[2]{\mathrel{\mathop{\longrightarrow}\limits^{#1\hfill}_{\hfill#2}}}
\nc{\pl}{\cal{PLF}}
\nc{\tr}{\cal{RTF}}
\nc{\im}{\mathrm{Im}}
\nc{\ff}{\cal{F}_\Omega}
\nc{\tm}{T_\Omega}
\nc{\calp}{\cal{P}}
\makeatletter
\nc\dd{\@ifnextchar'{\ddA}{\ddB}}
\def\ddA'#1;{\lhd'_{#1\,}}
\def\ddB#1;{\lhd_{#1\,}}
\nc{\pbt}{\mathrm{PBT}}
\nc{\ad}{\mathrm{ad}}

\makeatother

\begin{document}

\title[A Hopf algebra of finite topological quandles ]{A twisted Hopf algebra of finite topological quandles}
\thispagestyle{empty}
\author{Mohamed Ayadi}
\address{Laboratoire de Math\'ematiques Nicolas Oresme,
CNRS--Universit\'e de Caen Normandie,
Esp. de la Paix, 14000 Caen, France and University of Sfax, Faculty of Sciences of Sfax,
LAMHA, route de Soukra,
3038 Sfax, Tunisia.}
\email{mohamedayadi763763@gmail.com}
\author{Dominique Manchon}
\address{Laboratoire de Math\'ematiques Blaise Pascal,
CNRS--Universit\'e Clermont-Auvergne,
3 place Vasar\'ely, CS 60026,
F63178 Aubi\`ere, France}
\email{Dominique.Manchon@uca.fr}

\tikzset{
			stdNode/.style={rounded corners, draw, align=right},
			greenRed/.style={stdNode, top color=green, bottom color=red},
			blueRed/.style={stdNode, top color=blue, bottom color=red}
		} 
		
%\maketitle
	\begin{abstract}
This paper describes some algebraic properties of the species of finite topological quandles. We construct two twisted bialgebra structures on this species, one of the first kind and one of the second kind. The obstruction for the structure to match the double twisted bialgebra axioms is explicitly described.
	\end{abstract}
	
\keywords{Quandles, Finite topological spaces, Species, Bialgebras.}
\subjclass[2020]{57K12, 16T05, 16T10, 16T30.}
\maketitle
\tableofcontents
	%%%%
	%%%%%
	\section{Introduction}
%A quandle is a mathematical structure that describes the algebraic relationship between the elements of a set. Specifically,
A quandle is a set $Q$ with a binary operation $\lhd : Q \times Q \longrightarrow Q$ satisfying the three
axioms
\begin{itemize}
    \item (i) for every $a \in Q$, we have $a \lhd a = a$,
    \item (ii) for every pair $a, b \in Q$ there is a unique $c \in Q$ such that $a = c \lhd b$, and
    \item (iii) for every $a, b, c \in Q$, we have $(a \lhd b) \lhd c = (a \lhd c) \lhd (b \lhd c)$.
\end{itemize}
These three conditions that define a quandle originate from the axiomatization of the three reidemeister moves on knot diagrams. Quandles are algebraic structures that have various applications in knot theory and related fields. Two typical examples of quandles are the conjugation quandle and the core quandle. The conjugation quandle is derived from any group $(G, \circ)$, where the binary operation is given by conjugation, i.e. $x\lhd y=y^{-1}\circ x\circ y$. The core quandle, on the other hand, is another quandle derived from any group $(G, \circ)$, with the binary operation defined by $x\lhd y=x\circ y^{-1}\circ x$. Both of these quandles are of great importance in knot theory and have been studied extensively.  For more on quandles, see \cite{ Matveev, Joyce, Elh2, Elhamdadi, Yetter}.\\

%a quandle is a set equipped with a binary operation that satisfies three axioms: 
%\begin{itemize}
%    \item (i) the operation is self-distributive,
 %   \item (ii) it has an identity element, and
 %   \item (iii) every element has a unique inverse with respect to the operation.\\
%\end{itemize}
%(i) the operation is self-distributive, (ii) it has an identity element, and (iii) every element has a unique inverse with respect to the operation.\\
By Alexandroff’s theorem \cite{acg.Alex, acg..12}, for any finite set $X$, there is a bijection between topologies on $X$ and quasi-orders on $X$, where a quasi-order $\le$ in $X$ is a reflexive and transitive relation, not necessarily antisymmetric. For any $x,y\in X$, we write $x\le_{\mathcal T} y$ whenever any $\mathcal T$-open subset containing $x$ also contains $y$, and we note $x\sim_{\mathcal T} y$ whenever both $x\le_{\mathcal T}y$ and $y\le_{\mathcal T}x$ hold. More on finite topological spaces can be found in \cite{Moh. twisted, Moh. Doubling, acg15, acg16}.\\

 Given two topologies $\mathcal{T}$ and $\mathcal{T}'$ on a finite set $X$, we say that $\mathcal{T}'$ is finer than $\mathcal{T}$, denoted by $\mathcal{T}'\prec \mathcal{T}$, if every open subset of $\mathcal{T}$ is also an open subset of $\mathcal{T}'$. This is equivalent to saying that for any $x,y\in X$, if $x\le_{\mathcal{T}'} y$, then $x\le_{\mathcal{T}} y$. The quotient $\mathcal{T}/\mathcal{T}'$ of these two topologies is defined as follows: the associated quasi-order relation, denoted by $\le_{\mathcal{T}/\mathcal{T}'}$, is the transitive closure of the relation $\mathcal{R}$, which is defined by $x\mathcal{R} y$ if and only if $x\leq_{\mathcal{T}} y$ or $y\leq_{\mathcal{T}'} x$. F. Fauvet, L. Foissy and D. Manchon in \cite{acg10} define a relation noted $\oprec$ on the set of topologies in $X$ as follows: $\mathcal{T}^{\prime}\oprec \mathcal{T}$ if and only if
\begin{itemize}
	\item $\mathcal{T}^{\prime}\prec \mathcal{T}$,
	\item $\mathcal{T}^{\prime}_{|Y}=\mathcal{T}_{|Y}$ for any subset $Y\subset X$ connected for the topology $\mathcal{T}^{\prime}$,\
	\item for any $x, y \in X$,
	\begin{equation*}
	 x \sim_{\mathcal{T}/ \mathcal{T}^{\prime}} y \iff x \sim_{\mathcal{T}^{\prime}/ \mathcal{T}^{\prime}} y.
	 \end{equation*}
\end{itemize}
% This construction is often used in algebraic geometry and topology to build new spaces by collapsing certain subsets of a given space.\\

Let $(Q, \le)$ be a topological space equipped with a continuous map $\mu : Q \times Q \longrightarrow Q$ , denoted by  $\mu(a, b) = a \lhd b$, such that for every $b\in Q$ the mapping $R_b:a\mapsto a \lhd b$ is a homeomorphism of $(Q,\le)$. The space $Q$ (together with the map $\mu$ ) is called a topological quandle \cite{Rubinsztein} if it satisfies for all $a, b, c \in Q$
\begin{itemize}
    \item (i) $(a \lhd b) \lhd c=(a \lhd c) \lhd (b \lhd c)$,
    \item (ii) $a \lhd a=a$.
\end{itemize}

A finite topological quandle is a topological quandle with a finite underlying set. The study of finite topological quandles is important because finite quandles arise naturally in the study of knots and links, and topological quandles provide a way to study the geometry of these structures \cite{CSK}.\\
%$\lhd : Q \times Q \longrightarrow Q$
%One important result in the study of finite topological quandles is the Orbit Decomposition Theorem, which states that any finite topological quandle can be decomposed into a disjoint union of orbits, where an orbit is a set of elements that are related to each other under the quandle operation. This decomposition allows one to study the structure of the quandle by focusing on its orbits.\\

%Finite topological quandles also have connections to other areas of mathematics, such as group theory and combinatorics. For example, the set of automorphisms of a %finite topological quandle forms a group, and the study of these groups can provide insight into the structure of the quandle. Additionally, finite topological quandles can %be used to construct combinatorial objects, such as quandle cocycles, which have applications in the study of knot invariants.\\

The species formalism, due to A. Joyal \cite{J1981, J1986}, is an important tool in combinatorics. The idea of a species is to formalize the notion of "combinatorial equivalence" between objects of a given type, so that one can study the properties of the objects without worrying about their particular representations. To be precise, a linear species is a contravariant functor from the category of finite sets with bijections to the category of vector spaces over a given field $\mathbf{k}$. Specifically, a linear species $\mathbb{E}$ assigns to any finite set $X$ a vector space $\mathbb{E}_X$ over $\mathbf{k}$, and assigns to any bijection $\sigma:X\to Y$ a linear isomorphism $\mathbb E_\sigma:\mathbb E_Y\to \mathbb E_X$, such that the awaited functorial properties hold. One important operation on linear species is the Cauchy tensor product, denoted by $\otimes$, which takes two species $\mathbb{E}$ and $\mathbb{F}$ and produces a new species $\mathbb{E}\otimes \mathbb{F}$ 
%The definition of $\mathbb{E}\otimes \mathbb{F}$ is given by the direct sum over all possible disjoint partitions of the input set $X$ into two subsets $Y$ and $Z$, i.e.,
%\begin{equation*}
%	(\mathbb{E}\otimes \mathbb{F})_X=\bigoplus_{Y\sqcup Z= X}\mathbb{E}_{Y}\otimes\mathbb{F}_{Z},
%	\end{equation*}
%where $\sqcup$ denotes the disjoint union of sets. The tensor product operation is a fundamental tool in the study of linear species and has numerous applications in algebra, topology, and combinatorics.\\
%%%%%%%%%%%%%%%%%%%%%%%%%%%%%%%%
%The definition of $\mathbb{E}\otimes \mathbb{F}$ is
defined by   
\begin{itemize}
\item $(\mathbb{E}\otimes \mathbb{F})_A=\bigoplus \limits_{\underset{}{I\subseteq A}}\mathbb{E}_I\otimes \mathbb{F}_{A\backslash I}$\index{$(\mathbb{E}\otimes \mathbb{F})_A=\bigoplus \limits_{\underset{}{I\subseteq A}}\mathbb{E}_I\otimes \mathbb{F}_{A\backslash I}$},
\item for any bijection $\sigma : B\longrightarrow A$, 
\[(\mathbb{E}\otimes \mathbb{F})(\sigma) : \begin{cases}
\bigoplus \limits_{\underset{}{J\subseteq B}}\mathbb{E}_J\otimes \mathbb{F}_{B\backslash J}\longrightarrow \bigoplus \limits_{\underset{}{I\subseteq B}}\mathbb{E}_I\otimes \mathbb{F}_{B\backslash I}\\
\hspace{1.28cm}x\otimes y \longmapsto \mathbb{E}(\sigma_{|I})(x)\otimes \mathbb{F}(\sigma_{|B \backslash I})(y).
\end{cases}\]
\end{itemize}
%%%%%%%%%%%%%%%%%%%%%%%
We also recall that, for any two linear species $\mathbb{E}$ and $\mathbb{F}$, their Hadamard tensor product is defined by \cite{acg11}:  
    \begin{itemize}
        \item $(\mathbb{E}\odot \mathbb{F})(A)=\mathbb{E}(A)\otimes \mathbb{F}(A)$\index{$(\mathbb{E}\odot \mathbb{F})(A)=\mathbb{E}(A)\otimes \mathbb{F}(A)$},
        \item for bijection $\sigma : B\longrightarrow A$, $(\mathbb{E}\odot \mathbb{F})(\sigma)=\mathbb{E}(\sigma)\otimes \mathbb{F}(\sigma)$.
    \end{itemize}

%Recall that a linear species is a contravariant functor from the category
%	of finite sets with bijections into the category of vector spaces (on some
%	field $\mathbf{k}$). The tensor product of two species  $\mathbb{E}$ and $\mathbb{F}$ is given by
%	\begin{equation*}
%	(\mathbb{E}\otimes \mathbb{F})_X=\bigoplus_{Y\sqcup Z= X}\mathbb{E}_{Y}\otimes\mathbb{F}_{Z},
%	\end{equation*}
%	where the notation $\sqcup $ stands for disjoint union.\\

The species $\mathbb{QT}$ of finite topological quandles describes finite topological quandles up to combinatorial equivalence. Specifically, the species $\mathbb{QT}$ is a contravariant functor from the category of finite sets with bijections to the category of vector spaces, which associates to each finite set S the linear span of all finite topological quandles with underlying set S, i.e.,
the species $\mathbb{QT}$ is defined by:
\begin{itemize}
        \item for any finite set
$A$, $\mathbb{QT}_A$ is the vector space freely generated by the topological quandle stuctures on $A$, i.e., $\mathbb{QT}_A= \hbox{span}(A, \lhd,\le)$,  where  $(A, \lhd)$  is a quandle and  $\le$  is a quasi-order compatible with $(A, \lhd)$,
        \item for any bijection $\sigma: B\longrightarrow A$, $\mathbb{Q}_\sigma$ sends the topological quandle $Q=(A, \lhd, \le)$ to the topological quandle $\mathbb{Q}_\sigma (Q)=(B, \blacktriangleleft, \le')$, where $\blacktriangleleft$ and $\le'$ are defined by relabeling.
    \end{itemize}
${}$\\

The present article is organized as follows: in Section \ref{matrix of a finite quandle}, we revisit some important results related to finite quandles. Specifically, we remind the reader of the method developed by B. Ho and S. Nelson in \cite{B. Ho and S. Nelson} to describe finite quandles with at most 5 elements. Section \ref{3} contains our main results: we construct an external coproduct $\Delta$ defined for all $(X, \mathcal{T} ,\lhd)\in \mathbb{QT}_X$ (where $X$ is a finite set) by:
      	\begin{align*}
	\Delta:\mathbb{QT}_X&\longrightarrow (\mathbb{QT} \otimes \mathbb{QT})_X=\bigoplus \limits_{Z\subset X}\mathbb{QT}_{Z}\otimes\mathbb{QT}_{X\backslash Z}\\  	
	(X, \mathcal{T},\lhd)&\longmapsto \sum \limits_{Y \hbox{\tiny{ subquandle of X }} }(Y, \mathcal{T}_{|Y}, \lhd)\otimes (X\backslash Y, \mathcal{T}_{|X\backslash Y}, \lhd^{X, Y}),
	\end{align*}
 with an explicit quandle structure $\lhd^{X, Y}$ on the complement $X\backslash Y$. We moreover define for any finite set $X$ an internal coproduct $\Gamma: \mathbb{QT}_X\longrightarrow (\mathbb{QT}\odot \mathbb{QT})_X=\mathbb{QT}_X\otimes \mathbb{QT}_X$ by, for all $(X, \mathcal{T}, \lhd) \in \mathbb{QT}_X$:
$$\Gamma(X, \mathcal{T}, \lhd)=\sum \limits_{\underset{\mathcal{T}^{\prime} \hbox{ \tiny{is a Q-compatible}}}{\mathcal{T}^{\prime}\soprec  \mathcal{T} }}(X, \mathcal{T}^{\prime}, \lhd)\otimes (X, \mathcal{T}/ \mathcal{T}^{\prime}, \lhd).$$
%	$$\Gamma(X, \mathcal{T}, \lhd)=\sum \limits_{{\mathcal{T}^{\prime }\soprec \mathcal{T}}\atop \scriptstyle{\mathcal{T}^{\prime} \hbox{ \tiny{is a Q-compatible}}}}(X, \mathcal{T}^{\prime}, \lhd)\otimes (X, \mathcal{T}/ \mathcal{T}^{\prime}, \lhd).$$
 
 It indeed turns out that the quandle structure is compatible with both topologies $\mathcal T'$ and $\mathcal T/\mathcal T'$. The associative product $m$ of two topological quandles structures on $X$ and $Y$ respectively is given by the disjoint union of the topological spaces and the quandle structures involved: the action of elements of $X$ on $Y$ (and vice-versa) is trivial.
 % in $\mathbb{QT}$ by $m : \mathbb{QT}_{X_1} \otimes \mathbb{QT}_{X_2} \longrightarrow \mathbb{QT}_{X_1 \sqcup X_2}$, defined for all $(X_1, \mathcal{T}_1, \lhd_1)\in \mathbb{QT}_{X_1}$, $(X_2, \mathcal{T}_2, \lhd_2)\in \mathbb{QT}_{X_2}$, $m((X_1, \mathcal{T}_1, \lhd_1)\otimes (X_2, \mathcal{T}_2, \lhd_2))=(X_1\sqcup X_2, \mathcal{T}_1\mathcal{T}_2, \widetilde{\lhd})$, where $\mathcal{T}_1\mathcal{T}_2$ is the disjoint union topology characterised by $Y\in \mathcal{T}_1\mathcal{T}_2$ if and only if $Y\cap X_1\in \mathcal{T}_1$ and $Y\cap X_2\in \mathcal{T}_2$, and the operation $\widetilde{\lhd}$ defined by:
% \begin{itemize}
%\item $a\widetilde{\lhd} b=a\lhd_1 b$, \hbox{ for all } $a, b\in X_1$,
%\item $a\widetilde{\lhd} b=a\lhd_2 b$, \hbox{ for all } $a, b\in X_2$,
%\item $a\widetilde{\lhd} b=b$, \hbox{ for all } $a\in X_1, b\in X_2$,
%\item $a\widetilde{\lhd} b=a$, \hbox{ for all } $a\in X_2, b\in X_1$.
%\end{itemize}
We prove that $(\mathbb{QT}, m, \Delta)$ is a commutative connected twisted bialgebra and $(\mathbb{QT}, m, \Gamma)$ is a commutative connected twisted bialgebra on the second kind \cite{acg11}. %It turns out that only the internal coproduct $\Gamma$ is counital.\\
%%%%%%%%%%%%%%%%%%%%%%%%%%%%%%
%%%%%%%%%%%%%%%%%%%%%%%%%%%%%%%%%%%%%%
Finally, we define a map
	$$\xi : \mathbb{QT}_X \otimes (\mathbb{QT}\otimes \mathbb{QT})_X \longrightarrow \mathbb{QT}_X \otimes (\mathbb{QT}\otimes \mathbb{QT})_X$$
by:
\begin{equation*}\label{eq:xi}
	 \xi\big((X, \mathcal{T}, \lhd)\otimes (Y, \mathcal{T}_1, \lhd_1)\otimes (X\backslash Y, \mathcal{T}_2, \lhd_2)  \big)=(X, \mathcal{T}, \widetilde{\lhd} )\otimes (Y, \mathcal{T}_1,\lhd_1)\otimes (X\backslash Y, \mathcal{T}_2, \lhd_2)
	 \end{equation*}

where the new quandle structure $\widetilde\lhd$ is explicitly given, such that the coproduct $\Gamma $ and the map $\xi $ make the following diagram commute:
 $$
\xymatrix{
\mathbb{QT}_X \ar[rr]^\Gamma \ar[d]_{\Delta} && \mathbb{QT}_X \otimes \mathbb{QT}_X \ar[d]^{Id \otimes \Delta}\\
	(\mathbb{QT}\otimes\mathbb{QT})_X \ar[d]_{\Gamma \otimes \Gamma } && \mathbb{QT}_X \otimes (\mathbb{QT} \otimes \mathbb{QT})_X \ar[d]^{\xi}\\
\bigoplus \limits_{Y\subset X}\mathbb{QT}_Y \otimes \mathbb{QT}_Y \otimes \mathbb{QT}_{X\backslash Y} \otimes \mathbb{QT}_{X\backslash Y} \ar[rr]_{ m^{1,3} } && \mathbb{QT}_X \otimes (\mathbb{QT} \otimes \mathbb{QT})_X
}
$$
i.e.,
$$\xi \circ(Id\otimes \Delta )\circ \Gamma =m^{1,3} \circ(\Gamma \otimes \Gamma )\circ \Delta.$$
In other words, $\mathbb{QT}$ is nearly a twisted double bialgebra in the sense of \cite{acg11}, the defect being precisely described by the map $\xi$.
%%%%%%%%%%%%%%%%%%%%%%%%%%%%%%%%%%%%%%%%%%%%%%%%%%%
\section{Review of finite quandles}\label{matrix of a finite quandle}
%\textbf{Notation:}

Let $Q=\{x_1, x_2, . . . , x_n\}$ be a finite quandle with $n$ elements. B. Ho and S. Nelson in \cite{B. Ho and S. Nelson} defined the
matrix of $Q$, denoted $M_Q$, to be the matrix whose entry in row $i$ column $j$ is $x_i \lhd x_j$:\\
$$M_Q=	\begin{bmatrix}
x_1\lhd x_1 & x_1\lhd x_2 &...& x_1\lhd x_n\\
x_2\lhd x_1 & x_2\lhd x_2 &...& x_2\lhd x_n\\
.&.&...&.\\
.&.&...&.\\
.&.&...&.\\
x_n\lhd x_1 & x_n\lhd x_2 &...& x_n\lhd x_n
\end{bmatrix}
$$
\begin{examples}\cite{B. Ho and S. Nelson}
    For $Q=\{a, b, c \}$, the quandle matrices for quandles of order 3 are, up to permutations of the underlying three-element set:
   $$\begin{bmatrix}
a & a & a\\
b & b & b\\
c & c & c
\end{bmatrix}
, \hspace{1cm}
\begin{bmatrix}
a & c & b\\
c & b & a\\
b & a & c
\end{bmatrix}
, \hspace{1cm}
\begin{bmatrix}
a & a & a\\
c & b & b\\
b & c & c
\end{bmatrix}
$$ 
\end{examples}
\begin{definition}
Let $Q$ be a quandle. A subquandle $X \subset Q$ is a subset of $Q$ which is itself a quandle under $\lhd$. Let $Q$ be a quandle and $X \subset Q$ a subquandle. We say that $X$ is complemented in $Q$ or $Q$-complemented if $Q\backslash X$ is a subquandle of $Q$. 
\end{definition}
\noindent \textbf{Notation.} Let $(Q, \lhd)$ be a finite quandle, for $x'\in Q$, we note 

\begin{equation*}
			   \begin{split}
			         R_{x'}:Q &\longrightarrow Q\\
		              x &\longmapsto x\lhd x',
			    \end{split}
			     \ \ \ \ \ \ \ \ \ \ \ \ \ \ \ \hbox{and} \ \ \ \ \ \ \ \ \ \ \ \ \ \ \begin{split}
			         L_{x'}:Q& \longrightarrow Q\\
		         x& \longmapsto x'\lhd x.
			    \end{split}
			\end{equation*}
   	\begin{remark}\rm
	For any finite quandle $Q$ the following statements are equivalent:
	\begin{itemize}
	\item $(Q, \mathcal{T})$ is a topological quandle,
	\item $R_{x}$ is a homeomorphism and $L_{x}$ is a continuous map for any $x\in Q$,
	\item  for all $x, y, x', y' \in X$, if $x\le x'$ and $y\le y'$, we have $x\lhd y\le x'\lhd y'$.
	\end{itemize}           
	\end{remark}
%\begin{theorem}\cite{Nelson and Wong}
%Let $Q$ be a finite quandle. Then $Q$ may be written as
%$$Q = Q_1 \amalg Q_2 \amalg \cdot \cdot \cdot \amalg Q_n,$$
%where every $Q_i$ is $Q$-complemented and no proper subquandle of any $Q_i$ is $Q$-complemented. This
%decomposition is well-defined up to isomorphism; if $Q \approx Q'$, then in the decompositions
%$Q = Q_1 \amalg Q_2 \amalg \cdot \cdot \cdot \amalg Q_n,$ and $Q' = Q'_1 \amalg Q'_2 \amalg \cdot \cdot \cdot \amalg Q'_m,$
%we have then $n = m$ and (after reordering if necessary), $Q_i=Q'_j$.
%\end{theorem}
%\begin{remark}\cite{Nelson and Wong}
%     The decomposition of a finite quandle into orbits coincides with our notion of decomposition into Q-complemented subquandles; this follows from the observation that the orbits in Q are Q-complemented subquandles. Q-complemented subquandle decomposition then gives us a new perspective on the division of Q into disjoint orbits. 
%\end{remark}

%\textbf{Notation.} Let $Q=(X, \lhd)$ be a finite quandle, and $Y$ be a subquandle of $Q$, we noted
%$\lhd^{X, Y}:X\backslash Y\times X\backslash Y\longrightarrow X\backslash Y$ defined by, for all $a, b\in X\backslash Y$, $a\lhd^{X, Y} b= R_b^{\alpha (b)}(a)$, where $\alpha (b)=inf\{\alpha, R_{b|Y}^{\alpha}=Id_{|Y}\}$

%%%%%%%%%%%%%%%%%%%%%%%%%%%%%%%%%%%%%%%%%%%%%
\section{Algebraic structure of the linear species of finite topological quandles}\label{3}
Let $Q=(X, \lhd)$ be a finite quandle, and $Y$ be a subquandle of $Q$. Let
$$\lhd^{X, Y}:X\backslash Y\times X\backslash Y\longrightarrow X\backslash Y$$
be defined by $a\lhd^{X, Y} b= R_b^{\alpha (b)}(a)$, where $\alpha (b)=\hbox{inf}\{\alpha, R_{b|Y}^{\alpha}=Id_{|Y}\}$.
%Let $Q=(X, \lhd)$ be a finite quandle and let $Y$ be a subquandle of $Q$. We noted $\lhd^{X, Y} : X\backslash Y \times X\backslash Y \longrightarrow X\backslash Y$ defined by $a\lhd^{X, Y}b=R_b^{\alpha(b)}(a)$, where $\alpha(b)=inf\{\alpha, R_{b|Y}^{\alpha}=Id_{|Y} \}$.

\begin{proposition}\label{compl-quandle}
Let $Q=(X, \lhd)$ be a finite quandle. For any subquandle $Y$ of $Q$, the pair $(X\backslash Y, \lhd^{X, Y})$ is a quandle.
\end{proposition}
\begin{proof}
It is clear that, for all $a\in X\backslash Y$, $a\lhd^{X, Y}a=a$. Moreover for all $c \in X\backslash Y$, $R_c: X\times X \to X$ is a bijection, hence so is $R^{\alpha(c)}_c$.  Since $R^{\alpha(c)}_{c|Y}=Id_{|Y}$, we get that $R^{\alpha(c)}_c: X\backslash Y\times X\backslash Y \to X\backslash Y$ is a bijection as well.\\

For all $a, b, c\in X\backslash Y$, $(a\lhd^{X, Y}c)\lhd^{X, Y}(b\lhd^{X, Y}c)= R^{\alpha(c)}_c (a)\lhd^{X, Y} R^{\alpha(c)}_c (b)=R^{\alpha(R^{\alpha(c)}_c (b))}_{R^{\alpha(c)}_c (b)}\circ R^{\alpha(c)}_c (a)$ and $(a\lhd^{X, Y}b)\lhd^{X, Y} c=R^{\alpha(c)}_c \circ R^{\alpha(b)}_b(a)$. Using $(a\lhd b)\lhd c=(a\lhd c)\lhd (b\lhd c)$, i.e., $R_c \circ R_b (a)=R_{R_c(b)}\circ R_c(a)$, then for all $n \in \mathbb{N}$ we have, 
%%%%%%%%%%%%%%%%%%%%%%%%%%%%%%%%%%%%%%%%%%%%%%%%%%%%%%
\begin{align*}
	R^n_c \circ R_b&=R^{n-1}_c\circ R_c\circ R_b\\
	&=R^{n-1}_c\circ R_{R_c(b)}\circ R_c\\
        &=R^{n-2}_c\circ R_c \circ R_{R_c(b)}\circ R_c\\
        &=R^{n-2}_c\circ R_{R^2_c(b)}\circ R^2_c\\
        &\hspace{0.2cm}.\\
        &\hspace{0.2cm}.\\
        &\hspace{0.2cm}.\\
        &=R_c\circ R_{R^{n-1}_c(b)}\circ R^{n-1}_c\\
        &=R_{R^{n}_c(b)}\circ R^{n}_c,
	\end{align*}
 and for all $m \in \mathbb{N}$ we have,
 \begin{align*}
	R_c \circ R^m_b&=R_c\circ R_b\circ R^{m-1}_b\\
	&= R_{R_c(b)}\circ R_c\circ  R_b\circ R^{m-2}_b\\
        &= R_{R_c(b)}\circ R_{R_c(b)}\circ  R_c\circ R^{m-2}_b\\
        &= R^2_{R_c(b)}\circ R_c\circ R^{m-2}_b\\
        &\hspace{0.2cm}.\\
        &\hspace{0.2cm}.\\
        &\hspace{0.2cm}.\\
        &= R^{m-1}_{R_c(b)}\circ R_c\circ R_b.\\
         &= R^{m}_{R_c(b)}\circ R_c.\\
	\end{align*}
%%%%%%%%%%%%%%%%%%%%%%%%%%%%%%%%%%%%%%%%%%%
Henc,  for all $n, m \in \mathbb{N}$ we have
$$
	R^n_c \circ R^m_b=R^m_{R^n_c(b)}\circ R^n_c.
$$
 Then $R^m_{R^n_c(b)}=R^n_c \circ R^m_b\circ R^{-n}_c$. So, $(R^m_{R^n_c(b)})_{|Y}=\hbox{Id}_{Y}$ if and only if $(R^m_b)_{|Y}=\hbox{Id}_{Y}$. Hence $\alpha(b)=\alpha\big(R^{\alpha(c)}_c(b)\big)$. We therefore get
$$R^{\alpha(R^{\alpha(c)}_c (b))}_{R^{\alpha(c)}_c (b)}\circ R^{\alpha(c)}_c (a)=R^{\alpha(c)}_c \circ R^{\alpha(b)}_b(a).$$
Then, for all $a, b, c\in X\backslash Y$, $(a\lhd^{X, Y}b)\lhd^{X, Y} c=(a\lhd^{X, Y}c)\lhd^{X, Y}(b\lhd^{X, Y}c)$, which proves Proposition \ref{compl-quandle}.
\end{proof}

Let $X$ be any finite set and $\mathbb{Q}_X= \hbox{span}(X, \lhd)$ the vector space of quandles in $X$. We define the external coproduct $\Delta$ by:
	\begin{align}\label{coprod-ext}
	\Delta:\mathbb{Q}_X&\longrightarrow (\mathbb{Q} \otimes \mathbb{Q})_X=\bigoplus \limits_{Z\subset X}\mathbb{Q}_{Z}\otimes\mathbb{Q}_{X\backslash Z}\notag\\  	
	(X,\lhd)&\longmapsto \sum \limits_{Y \hbox{\tiny{ subquandle of X }} }(Y, \lhd)\otimes (X\backslash Y, \lhd^{X, Y}),
	\end{align}
 and we define an associative product $m$ in $\mathbb{Q}$ by $m : \mathbb{Q}_{X_1} \otimes \mathbb{Q}_{X_2} \longrightarrow \mathbb{Q}_{X_1 \sqcup X_2}$, defined for all $Q_1=(X_1, \lhd_1)\in \mathbb{Q}_{X_1}$, $Q_2=(X_2, \lhd_2)\in \mathbb{Q}_{X_2}$, $m(Q_1\otimes Q_2)=(X_1\sqcup X_2, \widetilde{\lhd})$, where 
 \begin{itemize}
\item $a\widetilde{\lhd} b=a\lhd_1 b$, \hbox{ for all } $a, b\in X_1$,
\item $a\widetilde{\lhd} b=a\lhd_2 b$, \hbox{ for all } $a, b\in X_2$,
\item $a\widetilde{\lhd} b=b$, \hbox{ for all } $a\in X_1, b\in X_2$,
\item $a\widetilde{\lhd} b=a$, \hbox{ for all } $a\in X_2, b\in X_1$.
\end{itemize}
\begin{examples}
$$
m\left(
\begin{bmatrix}
c & c & c\\
e & d & d\\
d & e & e
\end{bmatrix} \otimes \begin{bmatrix}
a & b \\
b & b 
\end{bmatrix}\right)= \begin{bmatrix}
c & c & c & c & c\\
e & d & d & d & d\\
d & e & e & e & e\\
a & a & e & a & b\\
b & b & e & b & b 
\end{bmatrix},\hspace{1cm}
m\left(
 \begin{bmatrix}
a & b \\
a & b 
\end{bmatrix}\otimes \begin{bmatrix}
c & e & d\\
e & d & c\\
d & c & e
\end{bmatrix}\right)= \begin{bmatrix}
a & b & a & a & a\\
a & b & b & b & b\\
c & c & c & e & d\\
d & d & e & d & c\\
e & e & d & c & e 
\end{bmatrix}  
$$
$$\Delta\left(\begin{bmatrix}
a & a & a\\
c & b & b\\
b & c & c
\end{bmatrix}\right)= \begin{bmatrix}
b
\end{bmatrix}\otimes \begin{bmatrix}
a & a \\
c & c 
\end{bmatrix}+ \begin{bmatrix}
a
\end{bmatrix}\otimes \begin{bmatrix}
b & b \\
c & c 
\end{bmatrix} +\begin{bmatrix}
c
\end{bmatrix}\otimes \begin{bmatrix}
a & a \\
b & b 
\end{bmatrix}+\begin{bmatrix}
b & b \\
c & c 
\end{bmatrix}\otimes \begin{bmatrix}
a
\end{bmatrix}$$
$$\Delta\left(\begin{bmatrix}
a & c & b\\
c & b & a\\
b & a & c
\end{bmatrix}\right)= \begin{bmatrix}
b
\end{bmatrix}\otimes \begin{bmatrix}
a & a \\
c & c 
\end{bmatrix}+ \begin{bmatrix}
a
\end{bmatrix}\otimes \begin{bmatrix}
b & b \\
c & c 
\end{bmatrix} +\begin{bmatrix}
c
\end{bmatrix}\otimes \begin{bmatrix}
a & a \\
b & b 
\end{bmatrix}$$
\end{examples}
 %%%%%%%%%%%%%%%%%%%%%%%%%%
 \begin{theorem}
	$(\mathbb{Q}, m, \Delta)$ is a commutative connected twisted bialgebra.
 \end{theorem}
\begin{proof}
Let $Q_1=(X_1, \lhd_1)\in \mathbb{Q}_{X_1}$, $Q_2=(X_2, \lhd_2)\in \mathbb{Q}_{X_2}$ and $Q_3=(X_3, \lhd_3)\in \mathbb{Q}_{X_3}$, we have:\\
$m(m\otimes Id)(Q_1\otimes Q_2\otimes Q_3)=(X_1\sqcup X_2\sqcup X_3, \lhd)$ where:
\begin{itemize}
\item $a\lhd b=a\lhd_1 b$, \hbox{ for all } $a, b\in X_1$,
\item $a\lhd b=a\lhd_2 b$, \hbox{ for all } $a, b\in X_2$,
\item $a\lhd b=b$, \hbox{ for all } $a\in X_1, b\in X_2$,
\item $a\lhd b=a$, \hbox{ for all } $a\in X_2, b\in X_1$,
\item $a\lhd b=a\lhd_3 b$, \hbox{ for all } $a, b\in X_3$,
\item $a\lhd b=b$, \hbox{ for all } $a\in X_1\sqcup X_2, b\in X_3$,
\item $a\lhd b=a$, \hbox{ for all } $a\in X_3, b\in X_1\sqcup X_2$.
\end{itemize}
On the other hand,\\
$m(m\otimes Id)(Q_1\otimes Q_2\otimes Q_3)=(X_1\sqcup X_2\sqcup X_3, \overline{\lhd})$ where:
\begin{itemize}
\item $a\overline{\lhd} b=a\lhd_1 b$, \hbox{ for all } $a, b\in X_1$,
\item $a\overline{\lhd} b=a\lhd_2 b$, \hbox{ for all } $a, b\in X_2$,
\item $a\overline{\lhd} b=a\lhd_3 b$, \hbox{ for all } $a, b\in X_3$,
\item $a\overline{\lhd} b=a$, \hbox{ for all } $a\in X_2, b\in X_3$,
\item $a\overline{\lhd} b=a$, \hbox{ for all } $a\in X_3, b\in X_2$,
\item $a\overline{\lhd} b=b$, \hbox{ for all } $a\in X_1, b\in X_2\sqcup X_3$,
\item $a\overline{\lhd} b=a$, \hbox{ for all } $a\in X_2\sqcup X_3, b\in X_1$.
\end{itemize}
So, $m(m\otimes \hbox{Id})(Q_1\otimes Q_2\otimes Q_3)=(X_1\sqcup X_2\sqcup X_3, \lhd)=(X_1\sqcup X_2\sqcup X_3, \overline{\lhd})=m(m\otimes \hbox{Id})(Q_1\otimes Q_2\otimes Q_3)$.\\

%%%%%%%%%%%%%%%%%%
\noindent We have for any finite quandle  $Q=(X, \lhd)$:
$$(\Delta \otimes \hbox{Id})\Delta (X, \lhd)=\sum \limits_{{ \hbox{ \tiny{Y subquandle of X}}}\atop \scriptstyle{Z \hbox{ \tiny{subquandle of Y}}}}(Z, \lhd)\otimes (Y\backslash Z, \lhd^{Y, Z})\otimes (X\backslash Y, \lhd^{X, Y}).$$
On the other hand, we have
$$(\hbox{Id}\otimes \Delta)\Delta (X, \lhd)=\sum \limits_{{ \hbox{ \tiny{Z subquandle of X}}}\atop \scriptstyle{U \hbox{ \tiny{subquandle of }}X\backslash Y}}(Z, \lhd)\otimes (U, \lhd^{X, Z})\otimes \big((X\backslash Z)\backslash U, \lhd^{X\backslash Z, U}\big).$$
The property of coassociativity can be derived from a straightforward observation, namely that the map $(Z, Y) \mapsto (Z, Y\backslash Z)$ is a bijection. This map takes pairs $(Z, Y)$ where $Y$ is a subquandle of $X$ and $Z$ is a subquandle of $Y$, and maps them onto pairs $(Z, U)$ where $Z$ is a subquandle of $X$ and $U$ is a subquandle of $X\backslash Z$. The inverse of this map is given by $(Z, U) \mapsto (Z, Z\sqcup U)$. As a result, $\lhd^{Y, Z}=\lhd^{X, Z}$ and $\lhd^{X, Y}=\lhd^{X\backslash Z, U}$. Finally, we show immediately that
		$$\Delta \circ m\big( (X_1, \lhd_1)\otimes(X_2, \lhd_2)\big) =m^{23}\big(\Delta(X_1, \lhd_1)\otimes\Delta (X_2, \lhd_2)\big).$$
\end{proof}
\begin{corollary}\label{corollary}
Let $X$ be any finite set and $\mathbb{QT}_X= \hbox{span}(X, \mathcal T,\lhd)$,  where  $(X, \lhd)$  is a quandle and  $\mathcal T$  is a topology compatible with $(X, \lhd)$.
Let $m$ the product defined in $\mathbb{QT}$ by
$$m\big((X_1, \mathcal{T}_1, \lhd_1))\otimes (X_2, \mathcal{T}_2, \lhd_2)\big)=(X_1\sqcup X_2, \mathcal{T}_1\mathcal{T}_2, \widetilde{\lhd})$$
and let $\Delta$ the coproduct defined by \eqref{coprod-ext}. Then $(\mathbb{QT}, m, \Delta)$ is a commutative connected twisted bialgebra.	
\end{corollary}
\begin{proof}
It suffices to show the coassociativity of coproduct $\Delta$ in the species of topological quandles $\mathbb{QT}$.
Let $X$ be a finite set and $Q=(X, \mathcal{T}, \lhd)\in \mathbb{QT}_X$, we have
\begin{align*}
   (\Delta \otimes \hbox{Id})\Delta (X, \mathcal{T}, \lhd)&=\sum \limits_{{\hbox{ \tiny{Y subquandle of X}}}\atop \scriptstyle{Z \hbox{ \tiny{subquandle of Y}}}}(Z, \mathcal{T}_{|Z}, \lhd)\otimes (Y\backslash Z, \mathcal{T}_{|Y\backslash Z}, \lhd^{Y, Z})\otimes (X\backslash Y, \mathcal{T}_{|X\backslash Y}, \lhd^{X, Y}).
\end{align*}
On the other hand,
$$(\hbox{Id}\otimes \Delta)\Delta (X,\mathcal{T} , \lhd)=\sum \limits_{{\hbox{ \tiny{Z subquandle of X}}}\atop \scriptstyle{U \hbox{ \tiny{subquandle of }}X\backslash Y}}(Z, \mathcal{T}_{|Z}, \lhd)\otimes (U, \mathcal{T}_{|U},\lhd^{X, Z})\otimes ((X\backslash Z)\backslash U, \mathcal{T}_{|(X\backslash Z)\backslash U} \lhd^{X\backslash Z, U}).$$
which proves Corollary \ref{corollary}.
\end{proof}
%\begin{examples}

%\end{examples}
%F. Fauvet, L. Foissy and D. Manchon in \cite{acg10} define an operation $\oprec$ on vector space $\mathbb{T}_X$ by, for all $\mathcal{T}^{\prime}, \mathcal{T}$ be two topologies on $X$, $\mathcal{T}^{\prime}\oprec \mathcal{T}$ defined by:
%\begin{itemize}
%	\item $\mathcal{T}^{\prime}\prec \mathcal{T}$,
%	\item such that $\mathcal{T}^{\prime}_{|Y}=\mathcal{T}_{|Y}$ for any subset $Y\subset X$ connected for the topology $\mathcal{T}^{\prime}$,\
%	\item such that for any $x, y \in X$,
%	\begin{equation}
%	 x \sim_{\mathcal{T}/ \mathcal{T}^{\prime}} y \iff x \sim_{\mathcal{T}^{\prime}/ \mathcal{T}^{\prime}} y.
%	 \end{equation}
%\end{itemize}
\begin{lemma}\label{lemme-sandwich}(\cite[Propostion 2.7]{acg10})
		Let $\mathcal{T}$ and $\mathcal{T}^{\prime \prime}$ be two topologies on $X$. If $ \mathcal{T}^{\prime \prime}\oprec \mathcal{T}$,
		then $\mathcal{T}^{\prime}\longmapsto \mathcal{T}^{\prime}/\mathcal{T}^{\prime \prime}$ is a bijection from the set of topologies $\mathcal{T}^{\prime}$ on $X$ such that $\mathcal{T}^{\prime \prime}\oprec \mathcal{T}^{\prime}\oprec \mathcal{T}$ , onto the set of topologies $\mathcal{U}$ on $X$ such that $\mathcal{U}\oprec \mathcal{T}/\mathcal{T}^{\prime \prime}$. 
	\end{lemma}
\begin{theorem}\label{Kebab}
	Let $Q=(X, \lhd)$ be a finite quandle and let $\mathcal{T}$ be a topology on $X$. For any $ \mathcal{T}^{\prime}\oprec \mathcal{T}$ we have:
 \begin{itemize}
\item[1-] if $\mathcal{T}$ and $\mathcal{T}^{\prime}$ are Q-compatible, then $\mathcal{T}/\mathcal{T}^{\prime}$ is Q-compatible,
\item[2-] if $\mathcal{T}$ and $\mathcal{T}/\mathcal{T}^{\prime}$ are Q-compatible, then $\mathcal{T}^{\prime}$ is Q-compatible.
 \end{itemize}
 %If $ \mathcal{T}^{\prime}\oprec \mathcal{T}$ and $\mathcal{T}$, $\mathcal{T}^{\prime}$ ,
%		then $\mathcal{T}^{\prime}\longmapsto \mathcal{T}^{\prime}/\mathcal{T}^{\prime \pr, defined for allime}$ is a bijection from the set of topologies $\mathcal{T}^{\prime}$ on $X$ such that $\mathcal{T}^{\prime \prime}\oprec \mathcal{T}^{\prime}\oprec \mathcal{T}$ , onto the set of topologies $\mathcal{U}$ on $X$ such that $\mathcal{U}\oprec \mathcal{T}/\mathcal{T}^{\prime \prime}$. 
	\end{theorem}
 \begin{proof}
 %%%%%%%%%%%%%%%ù
 Let $Q=(X, \lhd)$ be a finite quandle and let $\mathcal{T}^{\prime}\oprec \mathcal{T}$.\\
 1- If $\mathcal{T}$ and $\mathcal{T}^{\prime}$ are Q-compatible, then: for $x, x', y, y' \in X$ the hypotheses
$x\le_{\mathcal{T}/\mathcal{T}^{\prime}}x'$ and $y\le_{\mathcal{T}/\mathcal{T}^{\prime}}y'$ together imply that there exist $t_1,...,t_n, s_1,...,s_m \in X$ such that $x\mathcal{R}t_1\mathcal{R}t_2...\mathcal{R}t_n\mathcal{R}y$ and $x'\mathcal{R}s_1\mathcal{R}s_2...\mathcal{R}s_m\mathcal{R}y'$. Recall that $a\mathcal R b$ means ($a\le_{\mathcal T} b$ or $a\ge_{\mathcal T'} b$). First, we prove that $x\lhd y \mathcal{R} t_1\lhd s_1$ or $x\lhd y \mathcal{R} x\lhd s_1 \mathcal{R} t_1\lhd s_1$.\\
For $x\mathcal{R}t_1$, and $y\mathcal{R}s_1$, we have four possible cases:
\begin{itemize}
    \item First case; $x\le_{\mathcal{T}}t_1$, and $y\le_{\mathcal{T}} s_1$.
Since $\mathcal{T}$ is Q-compatible, then $x\lhd y\leq_{\mathcal{T}} t_1\lhd s_1$, hence $x\lhd y \mathcal{R} t_1\lhd s_1$.
    \item Second case; $x\ge_{\mathcal{T}'}t_1$, and $y\ge_{\mathcal{T}'} s_1$.
Since $\mathcal{T}'$ is Q-compatible, then $x\lhd y\ge_{\mathcal{T}'} t_1\lhd s_1$, hence $x\lhd y \mathcal{R} t_1\lhd s_1$.
    \item Third case; $x\le_{\mathcal{T}}t_1$, and $y\ge_{\mathcal{T}'} s_1$.
Since $R_{s_1}$ is continuous, then $R_{s_1}(x)\leq_{\mathcal{T}} R_{s_1}(t_1)$, so $x\lhd s_1 \mathcal{R} t_1\lhd s_1$  and since $L_{x}$ is continuous, then $L_x(y)\ge_{\mathcal{T}'}L_x(s_1)$, so $x\lhd y \ge_{\mathcal{T}'} x\lhd s_1$. therefore $x\lhd y \mathcal{R} x\lhd s_1 \mathcal{R} t_1\lhd s_1$.
    \item Fourth case; $x\ge_{\mathcal{T}'}t_1$, and $y\le_{\mathcal{T}} s_1$.
Since $R_{s_1}$ is continuous, then $R_{s_1}(x)\ge_{\mathcal{T}'} R_{s_1}(t_1)$, so $x\lhd s_1 \mathcal{R} t_1\lhd s_1$  and since $L_{x}$ is continuous, then $L_x(y)\le_{\mathcal{T}}L_x(s_1)$, so $x\lhd y \mathcal{R} x\lhd s_1$. Therefore $x\lhd y \mathcal{R} x\lhd s_1 \mathcal{R} t_1\lhd s_1$.\\
By induction we find that:\\
$(x\lhd y) \mathcal{R} (x\lhd s_1) \mathcal{R} (t_1\lhd s_1)\mathcal{R} (t_1\lhd s_2)\mathcal{R} (t_2\lhd s_2)\mathcal{R} (t_2\lhd s_3)... \mathcal{R} (t_{n-1}\lhd s_{n-1})\mathcal{R} (t_{n-1}\lhd s_{n})\mathcal{R} (t_{n}\lhd s_{n})\mathcal{R} (t_{n}\lhd y')\mathcal{R} (x'\lhd y')$. Therefore $\mathcal{T}/\mathcal{T}^{\prime}$ is Q-comatible.
\end{itemize}

 %%%%%%%%%%%%%%%%%%%%%%%%%%%%%%%%%%%%%
2- If $\mathcal{T}$ and $\mathcal{T}/\mathcal{T}^{\prime}$ are Q-compatible, then:
 for $x, x', y, y' \in X$, using that $\mathcal{T}^{\prime}\oprec \mathcal{T}$, we get that ($x\le_{\mathcal{T}'}x'$ and $y\le_{\mathcal{T}'}y'$) implies ($x\le_{\mathcal{T}}x'$ and $y\le_{\mathcal{T}}y'$). Using that $\mathcal{T}$ is Q-compatible, then
  $x\lhd y\le_{\mathcal{T}}x'\lhd y'$.\\
  On the other hand,
  ($x\le_{\mathcal{T}'}x'$ and $y\le_{\mathcal{T}'}y'$) implies ($x\sim_{\mathcal{T}/\mathcal{T}^{\prime}}x'$ and $y\sim_{\mathcal{T}/\mathcal{T}^{\prime}}y'$). Using that $\mathcal{T}/\mathcal{T}^{\prime}$ is Q-compatible, then $x\lhd y\sim_{\mathcal{T}/\mathcal{T}^{\prime}}x'\lhd y'$. So $x\lhd y$ and $x'\lhd y'$ are in the same connected component for the topology $\mathcal{T}^{\prime}$, then:  $x\lhd y\le_{\mathcal{T}}x'\lhd y'$ implies that  $x\lhd y\sim_{\mathcal{T}'}x'\lhd y'$ (because $\mathcal{T}^{\prime}_{|Y}=\mathcal{T}_{|Y}$ for any subset $Y\subset X$ connected for the topology $\mathcal{T}^{\prime}$).
 This proves that $\mathcal{T}^{\prime}$ is Q-compatible.
 \end{proof}
\noindent We define the internal coproduct $\Gamma$ for all $(X, \mathcal{T}, \lhd) \in \mathbb{QT}_X$ by:
$$\Gamma(X, \mathcal{T}, \lhd)=\sum \limits_{\underset{\mathcal{T}^{\prime} \hbox{ \tiny{is a Q-compatible}}}{\mathcal{T}^{\prime}\soprec  \mathcal{T} }}(X, \mathcal{T}^{\prime}, \lhd)\otimes (X, \mathcal{T}/ \mathcal{T}^{\prime}, \lhd).$$
%	$$\Gamma(X, \mathcal{T}, \lhd)=\sum \limits_{{\mathcal{T}^{\prime }\soprec \mathcal{T}}\atop \scriptstyle{\mathcal{T}^{\prime} \hbox{ \tiny{is a Q-compatible}}}}(X, \mathcal{T}^{\prime}, \lhd)\otimes (X, \mathcal{T}/ \mathcal{T}^{\prime}, \lhd).$$
\begin{theorem}
		$(\mathbb{QT}, m, \Gamma)$ is a commutative twisted bialgebra of the second kind.
	\end{theorem}
	\begin{proof}
 Let $X$ be a finite set, for $(X, \mathcal{T}, \lhd) \in \mathbb{QT}_X$, we have
 $$(\Gamma \otimes \hbox{Id})\Gamma(X, \mathcal{T}, \lhd)=\sum \limits_{\underset{\mathcal{T}^{\prime \prime}, \mathcal{T}^{\prime} \hbox{ \tiny{is a Q-compatible}}}{\mathcal{T}^{\prime \prime}\soprec\mathcal{T}^{\prime}\soprec  \mathcal{T} }}(X, \mathcal{T}^{\prime \prime}, \lhd)\otimes (X, \mathcal{T}^{\prime}/ \mathcal{T}^{\prime \prime}, \lhd)\otimes (X, \mathcal{T}/ \mathcal{T}^{\prime}, \lhd).$$
 %$$(\Gamma \otimes Id)\Gamma(X, \mathcal{T}, \lhd)=\sum \limits_{{\mathcal{T}^{\prime \prime}\soprec \mathcal{T}^{\prime }\soprec \mathcal{T}}\atop \scriptstyle{\mathcal{T}^{\prime \prime}, \mathcal{T}^{\prime} \hbox{ \tiny{is a Q-compatible}}}}(X, \mathcal{T}^{\prime \prime}, \lhd)\otimes (X, \mathcal{T}^{\prime}/ \mathcal{T}^{\prime \prime}, \lhd)\otimes (X, \mathcal{T}/ \mathcal{T}^{\prime}, \lhd).$$
 On the other hand, we have
 $$(\hbox{Id} \otimes \Gamma)\Gamma(X, \mathcal{T}, \lhd)=\sum \limits_{\underset{\mathcal{T}^{\prime \prime}, \mathcal{U} \hbox{ \tiny{are Q-compatible}}}{\mathcal{T}^{\prime \prime}\soprec  \mathcal{T},\, \mathcal{U}\soprec \mathcal{T}/ \mathcal{T}^{\prime \prime}  }}(X, \mathcal{T}^{\prime \prime}, \lhd)\otimes (X, \mathcal{U}, \lhd)\otimes (X, (\mathcal{T}/ \mathcal{T}^{\prime \prime})/\mathcal{U}, \lhd).$$
 The result then comes from Lemma \ref{lemme-sandwich} and Theorem \ref{Kebab}.
		Hence, $(\Gamma \otimes \hbox{Id})\Gamma =(\hbox{Id}\otimes \Gamma )\Gamma $, and consequently $\Gamma $ is coassociative. Finally we have directy:\\
		$$\Gamma \big( (X_1,\mathcal{T}_1, \lhd_1)(X_2, \mathcal{T}_2, \lhd_2)\big) =\Gamma(X_1,\mathcal{T}_1, \lhd_1)\Gamma(X_2, \mathcal{T}_2, \lhd_2).$$
 \end{proof}
\begin{examples}
For $(X, \lhd)=\begin{bmatrix}
a & a & a\\
c & b & b\\
b & c & c
\end{bmatrix}$ and $\mathcal{T}=\fcolorbox{white}{white}{
\scalebox{0.7}{
  \begin{picture}(54,67) (183,-149)
    \SetWidth{1.0}
    \SetColor{Black}
    \Vertex(194,-136){2}
    \Vertex(204,-136){2}
    \Text(191,-133)[lb]{\Large{\Black{$b$}}}
    \Text(202,-133)[lb]{\Large{\Black{$c$}}}
    \Arc(200,-132)(15.811,215,575)
    \Line(199,-116)(199,-105)
    \Vertex(199,-104){2}
    \Text(196,-100)[lb]{\Large{\Black{$a$}}}
  \end{picture}
}}$, then $(X, \mathcal{T}, \lhd)$ is a topological quandle and
$$\Gamma (\fcolorbox{white}{white}{
\scalebox{0.7}{
  \begin{picture}(54,67) (183,-149)
    \SetWidth{1.0}
    \SetColor{Black}
    \Vertex(194,-136){2}
    \Vertex(204,-136){2}
    \Text(191,-133)[lb]{\Large{\Black{$b$}}}
    \Text(202,-133)[lb]{\Large{\Black{$c$}}}
    \Arc(200,-132)(15.811,215,575)
    \Line(199,-116)(199,-105)
    \Vertex(199,-104){2}
    \Text(196,-100)[lb]{\Large{\Black{$a$}}}
  \end{picture}
}})=\fcolorbox{white}{white}{
\scalebox{0.7}{
  \begin{picture}(73,47) (183,-169)
    \SetWidth{1.0}
    \SetColor{Black}
    \Vertex(194,-154){2}
    \Vertex(204,-154){2}
    \Text(191,-150)[lb]{\Large{\Black{$b$}}}
    \Text(202,-150)[lb]{\Large{\Black{$c$}}}
    \Arc(200,-152)(15.811,215,575)
    \Vertex(223,-154){2}
    \Text(221,-150)[lb]{\Large{\Black{$a$}}}
  \end{picture}
}}\otimes \fcolorbox{white}{white}{
\scalebox{0.7}{
  \begin{picture}(54,67) (183,-149)
    \SetWidth{1.0}
    \SetColor{Black}
    \Vertex(194,-136){2}
    \Vertex(204,-136){2}
    \Text(191,-133)[lb]{\Large{\Black{$b$}}}
    \Text(202,-133)[lb]{\Large{\Black{$c$}}}
    \Arc(200,-132)(15.811,215,575)
    \Line(199,-116)(199,-105)
    \Vertex(199,-104){2}
    \Text(196,-100)[lb]{\Large{\Black{$a$}}}
  \end{picture}
}}+ \fcolorbox{white}{white}{
\scalebox{0.7}{
  \begin{picture}(54,67) (183,-149)
    \SetWidth{1.0}
    \SetColor{Black}
    \Vertex(194,-136){2}
    \Vertex(204,-136){2}
    \Text(191,-133)[lb]{\Large{\Black{$b$}}}
    \Text(202,-133)[lb]{\Large{\Black{$c$}}}
    \Arc(200,-132)(15.811,215,575)
    \Line(199,-116)(199,-105)
    \Vertex(199,-104){2}
    \Text(196,-100)[lb]{\Large{\Black{$a$}}}
  \end{picture}
}}\otimes \fcolorbox{white}{white}{
\scalebox{0.7}{
  \begin{picture}(65,49) (172,-166)
    \SetWidth{1.0}
    \SetColor{Black}
    \Arc(195,-143)(21.633,56,416)
    \Vertex(183,-151){2}
    \Vertex(194,-151){2}
    \Vertex(204,-152){2}
    \Text(180,-147)[lb]{\Large{\Black{$a$}}}
    \Text(191,-147)[lb]{\Large{\Black{$b$}}}
    \Text(202,-147)[lb]{\Large{\Black{$c$}}}
  \end{picture}
}}.$$
For $(X, \lhd)=\begin{bmatrix}
a & a & a & a\\
b & b & b & c\\
c & c & c & b\\
d & d & d & d
\end{bmatrix},$ and $\mathcal{T}=\mathcal{T}=\fcolorbox{white}{white}{
\scalebox{0.7}{
  \begin{picture}(56,37) (150,-205)
    \SetWidth{1.0}
    \SetColor{Black}
    \Vertex(172,-203){2}
    \Vertex(160,-203){2}
    \Vertex(153,-191){2}
    \Vertex(167,-191){2}
    \Line(160,-203)(154,-192)
    \Line(160,-204)(167,-191)
    \Text(150,-207)[lb]{\Large{\Black{$a$}}}
    \Text(177,-207)[lb]{\Large{\Black{$d$}}}
    \Text(147,-189)[lb]{\Large{\Black{$b$}}}
    \Text(171,-191)[lb]{\Large{\Black{$c$}}}
  \end{picture}
}}$
then $(X, \mathcal{T}, \lhd)$ is a topological quandle and
\begin{align*}
    \Gamma(\fcolorbox{white}{white}{
\scalebox{0.7}{
  \begin{picture}(56,37) (150,-205)
    \SetWidth{1.0}
    \SetColor{Black}
    \Vertex(172,-203){2}
    \Vertex(160,-203){2}
    \Vertex(153,-191){2}
    \Vertex(167,-191){2}
    \Line(160,-203)(154,-192)
    \Line(160,-204)(167,-191)
    \Text(150,-207)[lb]{\Large{\Black{$a$}}}
    \Text(177,-207)[lb]{\Large{\Black{$d$}}}
    \Text(147,-189)[lb]{\Large{\Black{$b$}}}
    \Text(171,-191)[lb]{\Large{\Black{$c$}}}
  \end{picture}
}})=\fcolorbox{white}{white}{
\scalebox{0.7}{
  \begin{picture}(56,37) (150,-205)
    \SetWidth{1.0}
    \SetColor{Black}
    \Vertex(172,-203){2}
    \Vertex(160,-203){2}
    \Vertex(153,-191){2}
    \Vertex(167,-191){2}
    \Line(160,-203)(154,-192)
    \Line(160,-204)(167,-191)
    \Text(150,-207)[lb]{\Large{\Black{$a$}}}
    \Text(177,-207)[lb]{\Large{\Black{$d$}}}
    \Text(147,-189)[lb]{\Large{\Black{$b$}}}
    \Text(171,-191)[lb]{\Large{\Black{$c$}}}
  \end{picture}
}}\otimes \fcolorbox{white}{white}{
\scalebox{0.7}{
  \begin{picture}(61,35) (106,-198)
    \SetWidth{1.0}
    \SetColor{Black}
    \Arc(128,-184)(13.038,122,482)
    \Vertex(122,-189){2}
    \Vertex(128,-189){2}
    \Vertex(134,-189){2}
    \Vertex(109,-189){2}
    \Text(105,-185)[lb]{\Large{\Black{$d$}}}
    \Text(118,-185)[lb]{\Large{\Black{$b$}}}
    \Text(126,-185)[lb]{\Large{\Black{$c$}}}
    \Text(132,-185)[lb]{\Large{\Black{$a$}}}
  \end{picture}
}}+ \fcolorbox{white}{white}{
\scalebox{0.7}{
  \begin{picture}(48,37) (156,-206)
    \SetWidth{1.0}
    \SetColor{Black}
    \Vertex(160,-202){2}
    \Vertex(160,-192){2}
    \Vertex(166,-202){2}
    \Vertex(174,-202){2}
    \Text(150,-204)[lb]{\Large{\Black{$a$}}}
    \Text(151,-190)[lb]{\Large{\Black{$b$}}}
    \Text(167,-200)[lb]{\Large{\Black{$c$}}}
    \Text(177,-200)[lb]{\Large{\Black{$d$}}}
    \Line(161,-202)(160,-193)
  \end{picture}
}}\otimes \fcolorbox{white}{white}{
\scalebox{0.7}{
  \begin{picture}(54,67) (183,-149)
    \SetWidth{1.0}
    \SetColor{Black}
    \Vertex(194,-136){2}
    \Vertex(204,-136){2}
    \Vertex(214,-136){2}
    \Text(191,-133)[lb]{\Large{\Black{$a$}}}
    \Text(202,-133)[lb]{\Large{\Black{$b$}}}
    \Text(212,-133)[lb]{\Large{\Black{$d$}}}
    \Arc(200,-132)(15.811,215,575)
    \Line(199,-116)(199,-105)
    \Vertex(199,-104){2}
    \Text(196,-100)[lb]{\Large{\Black{$c$}}}
  \end{picture}
}}+ \fcolorbox{white}{white}{
\scalebox{0.7}{
  \begin{picture}(48,37) (156,-206)
    \SetWidth{1.0}
    \SetColor{Black}
    \Vertex(160,-202){2}
    \Vertex(160,-192){2}
    \Vertex(166,-202){2}
    \Vertex(174,-202){2}
    \Text(150,-204)[lb]{\Large{\Black{$a$}}}
    \Text(151,-190)[lb]{\Large{\Black{$c$}}}
    \Text(167,-200)[lb]{\Large{\Black{$b$}}}
    \Text(177,-200)[lb]{\Large{\Black{$d$}}}
    \Line(161,-202)(160,-193)
  \end{picture}
}}\otimes \fcolorbox{white}{white}{
\scalebox{0.7}{
  \begin{picture}(54,67) (183,-149)
    \SetWidth{1.0}
    \SetColor{Black}
    \Vertex(194,-136){2}
    \Vertex(204,-136){2}
    \Vertex(214,-136){2}
    \Text(191,-133)[lb]{\Large{\Black{$a$}}}
    \Text(202,-133)[lb]{\Large{\Black{$c$}}}
    \Text(212,-133)[lb]{\Large{\Black{$d$}}}
    \Arc(200,-132)(15.811,215,575)
    \Line(199,-116)(199,-105)
    \Vertex(199,-104){2}
    \Text(196,-100)[lb]{\Large{\Black{$b$}}}
  \end{picture}
}}\\
&\hspace{-7cm} +\fcolorbox{white}{white}{
\scalebox{0.7}{
  \begin{picture}(58,39) (179,-166)
    \SetWidth{1.0}
    \SetColor{Black}
    \Vertex(183,-161){2}
    \Vertex(194,-161){2}
    \Vertex(204,-161){2}
    \Vertex(214,-161){2}
    \Text(180,-156)[lb]{\Large{\Black{$a$}}}
    \Text(191,-156)[lb]{\Large{\Black{$b$}}}
    \Text(202,-156)[lb]{\Large{\Black{$c$}}}
    \Text(212,-156)[lb]{\Large{\Black{$d$}}}
  \end{picture}
}}\otimes  \fcolorbox{white}{white}{
\scalebox{0.7}{
  \begin{picture}(56,37) (150,-205)
    \SetWidth{1.0}
    \SetColor{Black}
    \Vertex(172,-203){2}
    \Vertex(160,-203){2}
    \Vertex(153,-191){2}
    \Vertex(167,-191){2}
    \Line(160,-203)(154,-192)
    \Line(160,-204)(167,-191)
    \Text(150,-207)[lb]{\Large{\Black{$a$}}}
    \Text(177,-207)[lb]{\Large{\Black{$d$}}}
    \Text(147,-189)[lb]{\Large{\Black{$b$}}}
    \Text(171,-191)[lb]{\Large{\Black{$c$}}}
  \end{picture}
}}.
\end{align*}
\end{examples}

	\begin{theorem}\label{cointeraction}
	For any finite set $X$, let
	 $\xi : \mathbb{QT}_X \otimes (\mathbb{QT}\otimes \mathbb{QT})_X \longrightarrow \mathbb{QT}_X \otimes (\mathbb{QT}\otimes \mathbb{QT})_X$ be the map defined by:
	 \begin{equation}\label{eq:xi}
	 \xi\big((X, \mathcal{T}, \lhd)\otimes (Y, \mathcal{T}_1, \lhd_1)\otimes (X\backslash Y, \mathcal{T}_2, \lhd_2)  \big)=(X, \mathcal{T}, \widetilde{\lhd} )\otimes (Y, \mathcal{T}_1,\lhd_1)\otimes (X\backslash Y, \mathcal{T}_2, \lhd_2)
	 \end{equation}
	 where 
 \begin{itemize}
\item $a\widetilde{\lhd} b=a\lhd b$, \hbox{ for all } $a, b\in Y$,
\item $a\widetilde{\lhd} b=a\lhd^{X, Y} b$, \hbox{ for all } $a, b\in X\backslash Y$,
\item $a\widetilde{\lhd} b=b$, \hbox{ for all } $a\in Y, b\in X\backslash Y$,
\item $a\widetilde{\lhd} b=a$, \hbox{ for all } $a\in X\backslash Y, b\in Y$.
\end{itemize}
	The following diagram is commutative:
$$
\xymatrix{
\mathbb{QT}_X \ar[rr]^\Gamma \ar[d]_{\Delta} && \mathbb{QT}_X \otimes \mathbb{QT}_X \ar[d]^{Id \otimes \Delta}\\
	(\mathbb{QT}\otimes\mathbb{QT})_X \ar[d]_{\Gamma \otimes \Gamma } && \mathbb{QT}_X \otimes (\mathbb{QT} \otimes \mathbb{QT})_X \ar[d]^{\xi}\\
\bigoplus \limits_{Y\subset X}\mathbb{QT}_Y \otimes \mathbb{QT}_Y \otimes \mathbb{QT}_{X\backslash Y} \otimes \mathbb{QT}_{X\backslash Y} \ar[rr]_{ m^{1,3} } && \mathbb{QT}_X \otimes (\mathbb{QT} \otimes \mathbb{QT})_X
}
$$
i.e.,
$$\xi \circ(\hbox{\rm Id}\otimes \Delta )\circ \Gamma =m^{1,3} \circ(\Gamma \otimes \Gamma )\circ \Delta.$$
\end{theorem} 
\begin{proof}
Let $X$ be a finite set and $(X, \mathcal{T}, \lhd)\in \mathbb{TQ}_X$, we have 
\begin{align*}
\xi \circ(\hbox{\rm Id}\otimes \Delta )\circ \Gamma(X, \mathcal{T}, \lhd)&=\xi \circ(\hbox{\rm Id}\otimes \Delta )\left(\sum \limits_{\underset{\mathcal{T}^{\prime} \hbox{ \tiny{is Q-compatible}}}{\mathcal{T}^{\prime}\soprec  \mathcal{T} }}(X, \mathcal{T}^{\prime}, \lhd)\otimes (X, \mathcal{T}/ \mathcal{T}^{\prime}, \lhd)\right)\\
&=\xi\left( \sum \limits_{\underset{\hbox{ \tiny{ Y subquandle of X}}}{\mathcal{T}^{\prime }\soprec \mathcal{T},\, \mathcal{T}^{\prime} \hbox{ \tiny{ Q-compatible}}  }}(X, \mathcal{T}^{\prime}, \lhd)\otimes (Y, (\mathcal{T}/ \mathcal{T}^{\prime})_{|Y}, \lhd)\otimes (X\backslash Y, (\mathcal{T}/ \mathcal{T}^{\prime})_{|X\backslash Y}, \lhd^{X, Y}) \right)\\
&=\sum \limits_{\underset{\hbox{ \tiny{ Y subquandle of X}}}{\mathcal{T}^{\prime }\soprec \mathcal{T},\, \mathcal{T}^{\prime} \hbox{ \tiny{is a Q-compatible}}  }}(X, \mathcal{T}^{\prime}, \widetilde{\lhd})\otimes (Y, (\mathcal{T}/ \mathcal{T}^{\prime})_{|Y}, \lhd)\otimes (X\backslash Y, (\mathcal{T}/ \mathcal{T}^{\prime})_{|X\backslash Y}, \lhd^{X, Y}).
\end{align*}
On the other hand,
\begin{align*}
m^{1,3} \circ(\Gamma \otimes \Gamma )\circ \Delta (X, \mathcal{T}, \lhd)&=m^{1,3} \circ(\Gamma \otimes \Gamma )\left(\sum \limits_{Y \hbox{\tiny{ subquandle of X }} }(Y, \mathcal{T}_{|Y}, \lhd)\otimes (X\backslash Y, \mathcal{T}_{X\backslash Y}, \lhd^{X, Y})\right)\\
&\hspace{-3.3cm}=m^{1,3}\left(\sum \limits_{\underset{\mathcal{T}_1 \hbox{ \tiny{is a } } Q_1\hbox{\tiny{-com}}, \,\mathcal{T}_2 \hbox{ \tiny{is } } Q_2\hbox{\tiny{-comp.}}}{Y\hbox{\tiny{ subquandle of }} X,\, \mathcal{T}_1\soprec \mathcal{T}_{|Y},\, \mathcal{T}_2\soprec \mathcal{T}_{|X\backslash Y}  }}(Y, \mathcal{T}_1, \lhd)\otimes (Y, \mathcal{T}_{|Y}/\mathcal{T}_1, \lhd)\otimes (X\backslash Y, \mathcal{T}_2, \lhd^{X,Y} )\otimes (X\backslash Y, \mathcal{T}_{X\backslash Y}/\mathcal{T}_2, \lhd^{X,Y})
\right)\\
&=\hspace{-3.3cm}=\sum \limits_{\underset{\mathcal{T}_1 \hbox{ \tiny{is a } } Q_1\hbox{\tiny{-com}}, \,\mathcal{T}_2 \hbox{ \tiny{is } } Q_2\hbox{\tiny{-comp.}}}{Y\hbox{\tiny{ subquandle of }} X,\, \mathcal{T}_1\soprec \mathcal{T}_{|Y},\, \mathcal{T}_2\soprec \mathcal{T}_{|X\backslash Y}  }}(X, \mathcal{T}_1\mathcal{T}_2, \widetilde{\lhd})\otimes (Y, \mathcal{T}_{|Y}/\mathcal{T}_1, \lhd)\otimes (X\backslash Y, \mathcal{T}_{X\backslash Y}/\mathcal{T}_2, \lhd^{X,Y})\\
&=\hspace{-3.3cm}=\sum \limits_{\underset{\mathcal{T}_1 \hbox{ \tiny{is } } Q_1\hbox{\tiny{-comp.}}, \,\mathcal{T}_2 \hbox{ \tiny{is } } Q_2\hbox{\tiny{-comp.}}}{Y\hbox{\tiny{ subquandle of }} X,\, \mathcal{T}_1\soprec \mathcal{T}_{|Y},\, \mathcal{T}_2\soprec \mathcal{T}_{|X\backslash Y}  }}(X, \mathcal{T}_1\mathcal{T}_2, \widetilde{\lhd})\otimes (Y, (\mathcal{T}/\mathcal{T}_1\mathcal{T}_2)_{|Y}, \lhd)\otimes (X\backslash Y, (\mathcal{T}/\mathcal{T}_1\mathcal{T}_2)_{X\backslash Y}, \lhd^{X,Y}).
\end{align*}
%We use that, for all $\mathcal{T}_1$, resp. $\mathcal{T}_2$ be a topology on $X_1$, resp. $X_2$. let $\mathcal{T}^{\prime}_1\oprec \mathcal{T}_1$ and $\mathcal{T}^{\prime}_2\oprec \mathcal{T}_2$. Then $\mathcal{T}^{\prime}_1\mathcal{T}^{\prime}_2\oprec \mathcal{T}_1\mathcal{T}_2$.
%Conversely, any topology $\mathcal{T}$ on $X_1\sqcup X_2$ such that $\mathcal{T}\oprec \mathcal{T}_1\mathcal{T}_2$ can be written $\mathcal{T}^{\prime}_1\mathcal{T}^{\prime}_2$ with $\mathcal{T}_i = \mathcal{T}_{|X_i}$ for $i=1, 2$ and we have $\mathcal{T}^{\prime}_i\oprec \mathcal{T}_i$. Wich proves Theorem.\\

Here, $\mathcal{T}_1$ and $\mathcal{T}_2$ are topologies on $X_1$ and $X_2$ respectively. Let $\mathcal{T}^{\prime}_1\oprec \mathcal{T}_1$ and $\mathcal{T}^{\prime}_2\oprec \mathcal{T}_2$. Then, we can see that $\mathcal{T}^{\prime}_1\mathcal{T}^{\prime}_2\oprec \mathcal{T}_1\mathcal{T}_2$. Conversely, for any topology $\mathcal{T}$ on the disjoint union $X_1\sqcup X_2$ such that $\mathcal{T}\oprec \mathcal{T}_1\mathcal{T}_2$, we can write $\mathcal{T}^{\prime}_1\mathcal{T}^{\prime}_2$ with $\mathcal{T}_i = \mathcal{T}_{|X_i}$ for $i=1, 2$, and we have $\mathcal{T}^{\prime}_i\oprec \mathcal{T}_i$. This proves the theorem.

\end{proof}

We therefore notice that $\Gamma$ and $\Delta$ are not compatible, i.e. we do not get a double twisted bialgebra. The map $\xi$ above precisely accounts for the defect.

%\textbf{Appendix: remarks and open questions}\\
%\textbf{Acknowledgements}: The authors would like to thank Mohamed Elhamdadi for useful suggestions and comments.
	
%	\textbf{Conflicts of interest}: none

\newpage
%$\mathscr A$ $\mathscr A$ $\mathscr A$ $\mathscr A$ $\mathscr A$ $\mathscr A$ $\mathscr A$ $\mathscr A$ $\mathscr A$ $\mathscr A$ $\mathscr A$ $\mathscr A$ $\mathscr A$ $\mathscr A$ $\mathscr A$ $\mathscr A$	$\mathscr A$ $\mathscr A$ $\mathscr A$ $\mathscr A$ $\mathscr A$ $\mathscr A$ $\mathscr A$ $\mathscr A$ $\mathscr A$ $\mathscr A$ $\mathscr A$ $\mathscr A$ $\mathscr A$ 
%-----------------------------------------------------------------------------------------
%-----------------------------------------------{chapter}{Bibliography}------------------------------------------


\begin{thebibliography}{10} \addcontentsline{toc}{chapter}{Bibliography}
	%%%%%%%%%%%-A-%%%%%%%%%%%%%
		\bibitem{acg.Alex}{P. Alexandroff},
    \textsl{Diskrete Räume},
    Rec. Math. Moscou, n. Ser. \textbf{2}, No3, 501--519 (1937).
    
%    
\bibitem{Moh. twisted}{M. Ayadi},
	\textsl{Twisted pre-Lie algebras of finite topological spaces},
	Communications in algebra \textbf{50}, 2115--2138 (2022).
	
%
\bibitem{Moh. Doubling}{M. Ayadi, D. Manchon},
	\textsl{Doubling bialgebras of finite topologies},
	Letters in Mathematical Physics \textbf{111},1--23 (2021).
	
%
\bibitem{CSK}{J. S. Carter,  J. Scott, S. Kamada, M. Saito},
	\textsl{Surfaces in 4-space}, Chapter 5, Springer Science and Business Media (2012).
	
%	
\bibitem{Elhamdadi}{M. Elhamdadi},
	\textsl{Distributivity in quandles and quasigroups},
	Algebra, geometry and mathematical physics, Springer Proc. Math. Stat. \textbf{85}, Springer, Heidelberg, 325--340 (2014), 
MR3275946.

%
\bibitem{Elh2}{M. Elhamdadi, S. Nelson},
	\textsl{Quandles - an introduction to the algebra of knots},
	Student Mathematical Library \textbf{74}, Amer. Math. Soc. Providence, RI (2015).
	
%
\bibitem{acg11}{L. Foissy},
	\textsl{Twisted bialgebras, cofreeness and cointeraction},
	preprint, arXiv:1905.10199 (2019).
	
% 
\bibitem{acg10}{F. Fauvet, L. Foissy, D. Manchon},
	\textsl{The Hopf algebra of finite topologies and mould composition};
	Ann. Inst. Fourier \textbf{67} No3, 911--945 (2017).
	
%
 \bibitem{B. Ho and S. Nelson}{B. Ho and S. Nelson},
	\textsl{Matrices and Finite Quandles},
	Homology, Homotopy and Applications \textbf{7}, 197--208 (2005).
	
%	
\bibitem{J1981}A. Joyal,
\textsl{Une combinatoire des s\'eries formelles},
Advances in Math. \textbf{42}, 1--82 (1981).

%
\bibitem{J1986}A. Joyal,
\textsl{Foncteurs analytiques et esp\`eces de structures},
Lect. Notes in Math. \textbf{1234}, 126--159 (1986).

%
 \bibitem{Joyce}{D. Joyce},
	\textsl{A classifying invariant of knots, the knot quandle},
	J. Pure Appl. Algebra \textbf{23}  No1, 37--65 (1982).
	
%
 \bibitem{Matveev}{S. V. Matveev},
	\textsl{Distributive groupoids in knot theory},
	Mathematics of the USSR-Sbornik, \textbf{47} No1, 73--83 (1984).
	
%	\bibitem{Nelson and Wong}{S. Nelson and C. Wong},
%	\textsl{On the orbit decomposition of finite quandles};
%	Journal of Knot Theory and Its Ramifications, vol. 15, p. 761-772, 2006.\\

\bibitem{L. Pedro and R. Dennis}{L. Pedro and R. Dennis},
	\textsl{On finite racks and quandles},
	Communications in Algebra \textbf{34}, 371--406 (2006).
	
%
\bibitem{Rubinsztein}{R. L. Rubinsztein},
	\textsl{Topological quandles and invariants of links},
	Journal of knot theory and its ramifications \textbf{16}, 789--808 (2007).
	
%	
\bibitem{acg15}{A. K. Steiner},
	\textsl{The lattice of topologies: structure and complementation},
	Trans. Am. Math. Soc. \textbf{122}, 379--398 (1966).
	
%
\bibitem{acg..12}{R. E. Stong},
	\textsl{Finite topological spaces},
	Trans. Amer. Math. Soc. \textbf{123}, 325--340 (1966).
	
%	
\bibitem{acg16}{R. S. Vaidyanathaswamy},
	\textsl{Set topology},
	Chelsea, New-York (1960).
	
 %   
\bibitem{Yetter}{D. N. Yetter},
	\textsl{Quandles and monodromy},
	Journal of Knot Theory and its Ramifications \textbf{12},  523--541 (2003).
	
	

	

	
	
	
	
%	\bibitem{acg10}{F. Fauvet, L. Foissy, D. Manchon},
%	\textsl{The Hopf algebra of finite topologies and mould composition};
%	Ann. Inst. Fourier, Tome 67, No. 3 (2017), 911-945.
	

 
	

	


\end{thebibliography}
\end{document}